\documentclass[11pt]{article}
\usepackage{t1enc}
\usepackage{lmodern}
\usepackage[T2A,T1]{fontenc}
\usepackage[utf8]{inputenc}
\usepackage[russian,english]{babel}
\usepackage{amsmath,amssymb,amsfonts,amsthm,mathrsfs,textcomp,url,bbm,cancel,comment,enumitem,mathtools,chngcntr}
\usepackage{graphicx}
\usepackage{float}
\usepackage[all,ps]{xy}
\usepackage[colorlinks,urlcolor=cyan,citecolor=MidnightBlue,linkcolor=MidnightBlue]{hyperref}
\usepackage[dvipsnames]{xcolor}
\usepackage[margin=2.5cm]{geometry}
\usepackage{longtable}
\usepackage{tikz}
\makeatletter
\newcommand{\affil}[1]{\gdef\@affil{\textsc{#1}}}
\newcommand{\address}[1]{\gdef\@address{#1}}
\newcommand{\email}[1]{\gdef\@email{\url{#1}}}
\newcommand{\@endstuff}{\par\vspace{\baselineskip}\noindent\footnotesize
  \begin{tabular}{@{}l}
    \@affil\\
    \@address\\
    \textit{E-mail address:} \@email
  \end{tabular}}
\AtEndDocument{\@endstuff}
\makeatother

\usetikzlibrary{matrix,calc,decorations,positioning}
\pgfkeys{/tikz/.cd,
  alt double distance/.initial=5pt,
  alt double step/.initial=1pt,}
\pgfdeclaredecoration{double deco}{initial}
{
  \state{initial}[width=\pgfkeysvalueof{/tikz/alt double step},next state=cont] {
    \pgfmoveto{\pgfpoint{\pgfkeysvalueof{/tikz/alt double step}}{\pgfkeysvalueof{/tikz/alt double distance}/2}}
    \pgfpathlineto{\pgfpoint{0.3\pgflinewidth}{\pgfkeysvalueof{/tikz/alt double distance}/2}}
    \pgfpathmoveto{\pgfpoint{0.3\pgflinewidth}{-\pgfkeysvalueof{/tikz/alt double distance}/2}}
    \pgfpathlineto{\pgfpoint{1pt}{-\pgfkeysvalueof{/tikz/alt double distance}/2}}
    \pgfcoordinate{lastup}{\pgfpoint{1pt}{\pgfkeysvalueof{/tikz/alt double distance}/2}}
    \pgfcoordinate{lastdown}{\pgfpoint{1pt}{-\pgfkeysvalueof{/tikz/alt double distance}/2}}}
  \state{cont}[width=\pgfkeysvalueof{/tikz/alt double step}]{
    \pgfmoveto{\pgfpointanchor{lastup}{center}}
    \pgfpathlineto{\pgfpoint{\pgfkeysvalueof{/tikz/alt double step}}{\pgfkeysvalueof{/tikz/alt double distance}/2}}
    \pgfcoordinate{lastup}{\pgfpoint{\pgfkeysvalueof{/tikz/alt double step}}{\pgfkeysvalueof{/tikz/alt double distance}/2}}
    \pgfmoveto{\pgfpointanchor{lastdown}{center}}
    \pgfpathlineto{\pgfpoint{\pgfkeysvalueof{/tikz/alt double step}}{-\pgfkeysvalueof{/tikz/alt double distance}/2}}
    \pgfcoordinate{lastdown}{\pgfpoint{\pgfkeysvalueof{/tikz/alt double step}}{-\pgfkeysvalueof{/tikz/alt double distance}/2}}}
  \state{final}[width=0pt]
  { 
    \pgfmoveto{\pgfpointdecoratedpathlast}}}
\tikzset{alt double/.style={decorate,decoration=double deco}}
\allowdisplaybreaks

\setlist[enumerate]{itemsep=0mm}

\newcommand{\toverset}[2]{%
  \mathop{#2}\limits^{\vbox to -.1ex{\kern-0.4ex\hbox{$\scriptstyle #1$}\vss}}}
\newcommand{\tightoverset}[2]{%
  \mathop{#2}\limits^{\vbox to -.5ex{\kern-0.4ex\hbox{$\scriptstyle #1$}\vss}}}
\renewcommand{\underset}[2]{%
  \mathop{#2}\limits_{\vbox to -.5ex{\kern-1.6ex\hbox{$\scriptstyle #1$}\vss}}}
\newcommand{\tightunderset}[2]{%
  \mathop{#2}\limits_{\vbox to -.5ex{\kern-1.8ex\hbox{$\scriptstyle #1$}\vss}}}

\def\Z{\mathbb{Z}}

\def\R{\mathbb{R}}
\def\C{\mathbb{C}}

\def\RP{\mathbb{R}P}
\def\CP{\mathbb{C}P}

\def\AA{\mathscr{A}}

\def\FF{\mathscr{F}}

\def\into{\hookrightarrow}
\def\imto{\looparrowright}

\def\ol{\overline}

\def\la{\langle}
\def\ra{\rangle}

\def\id{\mathop{\rm id}}

\def\sw1{{w_1{\rm s}}}
\def\sc1{{c_1{\rm s}}}

\newenvironment{thmref}[1]%
{\phantomsection\par\addvspace{.5\baselineskip}\noindent\textbf{Theorem \ref{#1}.\ignorespaces}\begin{em}}%
  {\end{em}\par\addvspace{.5\baselineskip}}%
\newenvironment{prf}[1][\unskip]%
{\par\addvspace{.5\baselineskip}\noindent\textbf{Proof #1.\enspace\ignorespaces}}%
{~$\square$\par\addvspace{.5\baselineskip}}%
{\par\addvspace{.5\baselineskip}\noindent\textbf{Sketch of the proof #1.\enspace\ignorespaces}}%
{~$\square$\par\addvspace{.5\baselineskip}}%
{\par\addvspace{.5\baselineskip}\noindent\textit{Claim.\enspace\ignorespaces}\begin{em}}%
  {\end{em}\par\addvspace{.5\baselineskip}}%
{\par\addvspace{.5\baselineskip}\noindent\textit{Proof #1.\enspace\ignorespaces}}%
{~$\diamond$\par\addvspace{.5\baselineskip}}%
{\par\addvspace{.5\baselineskip}\noindent\textit{Remark.\enspace\ignorespaces}}%
{\lpar\addvspace{.5\baselineskip}}%
\newenvironment{subjclass}[1]%
{\par\noindent\begin{small}\textit{#1 Mathematics Subject Classification.\enspace\ignorespaces}}%
  {\end{small}\par}%
\newenvironment{key}%
{\par\noindent\begin{small}\textit{Key words and phrases.\enspace\ignorespaces}}%
  {\end{small}\par}%
{\par\addvspace{.5\baselineskip}\noindent\begin{em}}%
  {\end{em}\par\addvspace{.5\baselineskip}}%
{\begin{ex} ~ \vspace{-.5em}\begin{enumerate}}%
    {\end{enumerate}\end{ex}}%
{\begin{figure}[H]
    \begin{center}
      \centering\includegraphics[scale=0.4]{#1}
      \begin{changemargin}{2cm}{2cm}
        \caption{\footnotesize\hangindent=1.4cm #2}
      \end{changemargin}}%
    {\vspace{-1cm}
    \end{center}
  \end{figure}}%
{\begin{figure}[H]
    \begin{center}
      \centering\includegraphics[scale=0.35]{#1}
      \begin{changemargin}{2cm}{2cm}
        \footnotesize{#2}
      \end{changemargin}}%
    {\vspace{-1cm}
    \end{center}
  \end{figure}}%
\newtheorem{thm}{Theorem}%
\newtheorem{lemma}{Lemma}[subsection]%
\newtheorem{prop}[lemma]{Proposition}%
\newtheorem{crly}[lemma]{Corollary}%
\newtheorem{ques}[lemma]{Question}%
\theoremstyle{definition}
\newtheorem{defi}[lemma]{Definition}%
\newtheorem{rmk}[lemma]{Remark}%
\newtheorem{ex}[lemma]{Example}%
\newcommand{\ZeroRoman}[1]{%
\ifcase\value{#1}\relax 0\else\Roman{#1}\fi}

\newcounter{t}

\numberwithin{equation}{section}
\numberwithin{figure}{section}
\title{
  On quasi-holomorphic homotopies of immersions\\of 3-manifolds into 5-manifolds
}
\author{András Csépai}
\affil{HUN-REN Alfréd Rényi Institute of Mathematics}
\address{Reáltanoda utca 13--15, Budapest, H-1053 Hungary}
\email{csepai@renyi.hu}
\date{}
\begin{document}

\maketitle

\begin{abstract}
  The notion of a quasi-holomorphic homotopy of two immersions of a $3$-manifold into a $5$-manifold extends the notion of their regular homotopy by also allowing the homotopy to pass through instances of maps with an isolated singularity around which the path of the homotopy forms a cross-cap (complex Whitney umbrella). We describe the local form of such homotopies and explain connections with the theory of holomorphic map germs from $\C^2$ to $\C^3$. 
  Our main result is a complete a description of how the fundamental group of the complement of the image of an immersion of a $3$-manifold into a $5$-manifold (i.e. its knot group) changes under a quasi-holomorphic homotopy.
  As a corollary we will see that certain quasi-holomorphic homotopies of the standard embedding of the $3$-sphere into the $5$-sphere do not change its knot group.
\end{abstract}

\begin{subjclass}{2020}
  57R42 (primary); 32S05, 57K45, 57R45, 58K05, 58K15 (secondary)
\end{subjclass}

\begin{key}
  quasi-holomorphic map; knot group; homotopic immersions; holomorphic map germ
\end{key}

\section{Introduction}
\label{sec:intro}

A regular homotopy $f_s,~s\in[0,1]$ of two stable immersions $f_0,f_1\colon M^3\imto N^5$ is such that its trivial extension $F\colon M^3\times[0,1]\to N^5\times[0,1]$ defined by $F(x,s):=(f_s(x),s)$ is an immersion which can be assumed to be stable. We extend this to the notion of \textit{quasi-holomorphic homotopy} where instead of $F$ being a stable immersion, we only expect it to be a stable \textit{quasi-holomorphic} map in the sense of \cite{qhol}. (Stable) quasi-holomorphic maps are defined as smooth maps between real manifolds whose germs are real suspensions of (stable) holomorphic map germs in the appropriate local coordinates. The only stable holomorphic germ, which is not regular, in the dimensions of the map $F$ is the cross-cap (complex Whitney umbrella)
$$\C^2\to\C^3,(u,v)\mapsto(u^2,uv,v),$$
hence the notion of quasi-holomorphic homotopy extends the notion of regular homotopy by allowing $F$ to have isolated cross-cap points. 
Quasi-holomorphic homotopies of immersions naturally occur in the theory of holomorphic map germs from $\C^2$ to $\C^3$; see subsection \ref{ssec:motiv}.

An immersion of a $3$-manifold into a $5$-manifold under a quasi-holomorphic homotopy changes by four types of Reidemeister type moves, namely \textit{triple point moves}, \textit{cross-cap moves} and \textit{elliptic} and \textit{hyperbolic self-tangency moves}. We will thoroughly describe all of them in subsection \ref{ssec:moves}. Our main results concern the fundamental groups of the complements of the images of quasi-holomorphically homotopic immersions, which we also call their \textit{knot groups}. These are the following.

\begin{thm}
  \label{thm:noell}
  If $f_s\colon M^3\to N^5,~s\in[-1,1]$ is a quasi-holomorphic homotopy 
  without elliptic self-tangency moves, then the groups $\pi_1(N^5\setminus f_{-1}(M^3))$ and $\pi_1(N^5\setminus f_1(M^3))$ are isomorphic.
\end{thm}

\begin{thm}
  \label{thm:cyc}
  If $f_s\colon M^3\to N^5,~s\in[-1,1]$ is a quasi-holomorphic homotopy which has only one elliptic self-tangency move, then there are presentations of the groups $\pi_1(N^5\setminus f_{\pm1}(M^3))$ such that their generator sets coincide and their relation sets only differ in one commutator relation: for two generators $a,b$ one of the above groups possesses the relation $ab=ba$ and the other one does not.  
\end{thm}

We note that the above theorems are interesting even if $f_s,~s\in[-1,1]$ is a regular homotopy. Together they completely describe how a knot group changes under quasi-holomorphic homotopies: it only changes under elliptic self-tangency moves and all changes correspond to possibly altering the commutation of two generators. Studying these generators and the way the new relation appears yields further, weaker conditions under which quasi-holomorphic homotopies still do not change the knot group; see subsection \ref{ssec:crly}. In particular, this holds for quasi-holomorphic homotopies of the standard embedding $S^3\into S^5$ with only one elliptic self-tangency move. 


\subsection*{Organisation of the paper}

In section \ref{sec:pre} we define the notions used throughout the paper (see subsection \ref{ssec:def}), then motivate the introduction and study of quasi-holomorphic homotopies (see subsection \ref{ssec:motiv}). In section \ref{sec:form} we give formulae for the local forms of moves in quasi-holomorphic homotopies (see subsection \ref{ssec:moves}), then connect these with the complex Reidemeister moves (see subsection \ref{ssec:cplxmove}). Lastly, in section \ref{sec:pi1} we put the above main theorems in the context of knot groups (see subsection \ref{ssec:knotgrp}), prove them (see subsection \ref{ssec:prfmain}), then state some of their corollaries (see subsection \ref{ssec:crly}).

\subsection*{Acknowledgements}

I thank Gergő Pintér, Tamás Terpai, András Némethi and András Szűcs for valuable consultations on the topic of this paper. In particular, Pintér suggested the unicity part in the proof of proposition \ref{prop:qholformula}; and Terpai drew my attention to the monodromy (non-)action explained in remark \ref{rmk:monodromy}.

\section{Preliminaries}
\label{sec:pre}

\subsection{Definitions}
\label{ssec:def}

Throughout the paper all manifolds are assumed to be smooth and connected. We consider stable immersions of $3$-manifolds into $5$-manifolds and of $4$-manifolds into $6$-manifolds. A stable immersion $f\colon M^3\imto N^5$ has simple regular points almost everywhere and transverse double regular points along a curve $D(f)\subset M^3$ mapped to a curve $\widetilde D(f)\subset N^5$ by the double covering $f|_{D(f)}$. A stable immersion $F\colon P^4\imto Q^6$ has simple regular points almost everywhere, transverse double regular points along a surface $D(F)\subset P^4$ mapped to a surface $\widetilde D(F)\subset Q^6$ by the double covering $F|_{D(F)}$, and transverse triple points in a discrete set of points $T(F)\subset P^4$ in the closure of $D(F)$ in the source and $\widetilde T(F)\subset Q^6$ in the closure of $\widetilde D(F)$ in the target.

\begin{defi}
  \label{defi:qholhom}
  Two stable immersions $f_0,f_1\colon M^3\imto N^5$ will be called \textit{quasi-holomorphically homotopic} if there is a homotopy $f_s\colon M^3\to N^5,~s\in[0,1]$ connecting them such that the map $F\colon M^3\times[0,1]\to N^5\times[0,1],F(x,s):=(f_s(x),s)$ is a stable quasi-holomorphic map in the sense of \cite{qhol}, i.e. it is a stable immersion outside a discrete set of points $C(F)\subset M^3\times[0,1]$ and for each $p\in C(F)$ there are neighbourhoods $p\in U\subset M^3\times[0,1]$ and $F(p)\in V\subset N^5\times[0,1]$ and diffeomorphisms $U\approx\C^2$ and $V\approx\C^3$ such that $F|_U$ is the cross-cap $(u,v)\mapsto(u^2,uv,v)$ in these coordinates. In this case $f_s,~s\in[0,1]$ will be called a \textit{quasi-holomorphic homotopy}.
\end{defi}

\begin{rmk}
  For a stable quasi-holomorphic map $F$ as above, the union $D(F)\cup T(F)\cup C(F)$ is the closure of the surface $D(F)$ and the union $\widetilde D(F)\cup\widetilde T(F)\cup\widetilde C(F)$ (where $\widetilde C(F)=F(C(F))$) is the closure of the surface $\widetilde D(F)$. The restrictions $D(F)\to\widetilde D(F)$, $T(F)\to\widetilde T(F)$ and $C(F)\to\widetilde C(F)$ of $F$ are a double cover, a triple cover and a bijection respectively.
\end{rmk}

For a quasi-holomorphic homotopy $f_s\colon M^3\to N^5,~s\in[0,1]$ there is a finite number of para\-meters $s\in[0,1]$ such that the map $f_s\colon M^3\to N^5$ is not a stable immersion. We may assume that at each such exceptional parameter there is exactly one point of instability in $N^5\times[0,1]$. Then at each exceptional parameter one of the following three instabilities occurs: \textit{cross-cap move} where $F$ 
has a cross-cap point in $N^5\times\{s\}$; \textit{triple point move} where $F$ has a triple point in $N^5\times\{s\}$; and \textit{self-tangency move} where the double point surface $\widetilde D(F)$ of $F$ in $N^5\times[0,1]$ is not transverse to $N^5\times\{s\}$. Note that the restriction of the function $N^5\times[0,1]\to[0,1]$, mapping $N\times\{s\}$ to $s$, to the surface $\widetilde D(F)$ is a Morse function. Self-tangency moves occur precisely at the critical points of this function. Hence such a move corresponds to a handle attachment to the double point surface $\widetilde D(F)$ when moving from $0$ to $1$, which can be of Morse index $0$, $1$ or $2$. A self-tangency move of index $0$ or $2$ will be called \textit{elliptic}; an index-$1$ self-tangency move will be called \textit{hyperbolic}. These namings are consistent with the ones in \cite{3knots} where regular homotopies of immersions $S^3\imto S^5$ were studied. The above moves are similar to the Reidemeister moves occurring in the diagrams of knots $S^1\into S^3$ under isotopies. In fact, the cross-cap, triple point and hyperbolic self-tangency moves exactly correspond to complex Reidemeister moves; see remark \ref{rmk:mnb}.

\subsection{Motivations}
\label{ssec:motiv}

A (generic) regular homotopy of two immersions $M^3\imto N^5$ is naturally a quasi-holomorphic homotopy without cross-cap moves. The knottedness properties of $2$-codimensional immersions were analysed by many. Classically, an \textit{$n$-knot} is an embedding $S^n\into S^{n+2}$ up to isotopy (see e.g. \cite{kerv}); changing ``embedding'' to ``stable immersion'' and ``isotopy'' to ``homotopy through stable immersions'' is a natural generalisation of this (see e.g. \cite{3knots}). The reason for this is the following proposition and remark \ref{rmk:isot} related to it.

\begin{prop}
  \label{prop:isot}
  If $f_s\colon M^n\imto N^{n+2},~s\in[0,1]$ is a homotopy such that for all $s\in[0,1]$ the map $f_s$ is a stable immersion, then there is a diffeotopy $h_s\colon N^{n+2}\to N^{n+2},~s\in[0,1]$ of the identity of $N^{n+2}$ such that for all $s\in[0,1]$ we have $h_s\circ f_s=f_0$.
\end{prop}

\begin{rmk}
  \label{rmk:isot}
  If $f_0$ or $f_1$ above is an embedding, then $f_s,~s\in[0,1]$ is an isotopy. Indeed, if for some parameter $s$ the map $f_s$ had multiple points, then there would also be a parameter $s'$ such that the map $f_{s'}$ has non-transverse self-intersection. Hence $f_{s'}$ would not be stable.
\end{rmk}

\begin{prf}[of proposition \ref{prop:isot}]
  A stable immersion $f\colon M^n\imto N^{n+2}$ has $m$-tuple points in an $(n+2-2m)$-dimensional submanifold $\widetilde M_m(f)\subset N^{n+2}$ where the branches of $f$ meet transversely. To obtain the desired diffeotopy $h_s,~s\in[0,1]$, we apply the isotopy extension theorem recursively with respect to the dimension of the multiple point strata $\widetilde M_m(f_s)$.

  If $m$ is the largest number for which $\widetilde M_m(f_s)$ is non-empty, then the submanifolds $\widetilde M_m(f_s),~s\in[0,1]$ form the path of an isotopy between the embeddings $\widetilde M_m(f_0)\subset N^{n+2}$ and $\widetilde M_m(f_1)\subset N^{n+2}$. Hence the isotopy extension theorem yields a diffeotopy $h_s^m,~s\in[0,1]$ of the identity of $N^{n+2}$ for which we have $\widetilde M_m(h_s^m\circ f_s)=\widetilde M_m(f_0)$ for all $s$. Now the restriction of $h_s^m\circ f_s$ to the preimage of the complement of a tubular neighbourhood $T\subset N^{n+2}$ of $\widetilde M_m(f_0)$ is a stable immersion between manifolds with boundaries mapping boundary to boundary, and such that it has no $m$-tuple points. Now we can apply the above method for the stratum of $(m-1)$-tuple points to obtain an analogous diffeotopy $h_s^{m-1},~s\in[0,1]$ of the identity of $N^{n+2}\setminus T$ straightening the $(m-1)$-tuple point stratum. We trivially extend this to a diffeotopy of the whole manifold $N^{n+2}$, which we will still denote by $h_s^{m-1}$. Then we have
  $$\widetilde M_m(h_s^{m-1}\circ h_s^m\circ f_s)\cup\widetilde M_{m-1}(h_s^{m-1}\circ h_s^m\circ f_s)=\widetilde M_m(f_0)\cup\widetilde M_{m-1}(f_0)$$
  for all $s$. Continuing this method recursively we obtain the diffeotopy $h_s:=h_s^0\circ\ldots\circ h_s^m$.
\end{prf}

For fixed manifolds $M^3$ and $N^5$, paths in the space of stable immersions $M^3\imto N^5$ are homotopies through stable immersions. If we extend this space to the space of all immersions, then the paths will be the regular homotopies; many properties of these in the case $M^3=S^3,N^5=S^5$ were described by Ekholm \cite{3knots}. If we also allow maps $M^3\to N^5$ with isolated singularities corresponding to cross-cap moves, then we obtain quasi-holomorphic homotopies.

Quasi-holomorphic homotopies also appear as parts of stable quasi-holomorphic cobordisms, which were studied in \cite{qhol}. In fact, the following remark shows that these are the parts of cobordisms yet to be understood.

\begin{rmk}
  \label{rmk:cob}
  Suppose that the immersions $f_0\colon M^3_0\imto N^5$ and $f_1\colon M^3_1\imto N^5$ are stably quasi-holomorphically cobordant through a map $F\colon P^4\to N^5\times[0,1]$. Then for almost all parameters $s\in[0,1]$, the slice $f_s:=F|_{M_s}$, where $M_s:=F^{-1}(N^5\times\{s\})$, is a stable immersion of a $3$-manifold into $N^5$. For all exceptional parameters $s$, we can assume that exactly one of the following two options occurs: $M_s$ is a $3$-manifold but the map $f_s$ is not stable; or $M_s$ is a manifold outside of a single point and the restriction of $f_s$ to the manifold part is stable. Then there are small numbers $\varepsilon_s>0$ such that in the first case $f_r,~r\in[s-\varepsilon_s,s+\varepsilon_s]$ is a quasi-holomorphic homotopy; and in the second case, the manifold $M_{s+\varepsilon_s}$ is obtained from $M_{s-\varepsilon_s}$ by a Morse type handle attachment. The local forms of the changes that occur in the second case are moves that occur in embedded Morse theory; they were thoroughly described in \cite{borpow}. The local forms in the first case will be described in section \ref{sec:form}.
\end{rmk}

A natural occurrence of these objects is given by finitely $\AA$-determined holomorphic map germs $\Phi\colon(\C^2,0)\to(\C^3,0)$, i.e. germs that are stable immersions off the origin. The topology of such a germ only depends on its \textit{link immersion}, that is, its restriction $f\colon S^3\imto S^5$ to the preimage of a sufficiently small sphere around the origin (which is a sphere itself); see e.g. \cite{associmm}. The connection of these map germs with quasi-holomorphic homotopies is the following. 

\begin{prop}
  \label{prop:link}
  If $\Phi\colon(\C^2,0)\to(\C^3,0)$ is a finitely $\AA$-determined holomorphic map germ, then the link immersion $f\colon S^3\imto S^5$ of $\Phi$ is stably quasi-holomorphically cobordant 
  to the standard embedding $S^3\subset S^5$.
\end{prop}


\begin{prf}
  Let $\{\Phi_t\}_{t\in\C}$ be a stable deformation of $\Phi$, i.e. a family of holomorphic maps $\C^2\to\C^3$ such that we have $\Phi_0=\Phi$ and for all $t\ne0$ the maps $\Phi_t$ are stable and pairwise left-right equivalent. This exists and is essentially unique; see e.g. \cite{wallfin}. We may also assume (up to diffeomorphisms) that for some parameter $t\in\C$, the preimage of the closed unit ball $D^6\subset\C^3$ under $\Phi_t$ is the unit ball $D^4\subset\C^2$, and the restriction $\Phi_t|_{\C^2\setminus D^4}$ is just the cylinder over its restriction to $S^3=\partial D^4$. Then the restriction of $\Phi_t|_{D^4}\colon D^4\to D^6$ to the boundary is $f\colon S^3\imto S^5$.

  We can suppose that the origin in $D^6$ is a simple regular value of the map $\Phi_t$ and the preimage under $\Phi_t$ of the open $\varepsilon$-ball $B^6_\varepsilon\subset\C^3$ is the open $\varepsilon$-ball $B^4_\varepsilon\subset\C^2$. Moreover, we can also assume that the restriction $\Phi_t|_{B^4_\varepsilon}$ is the standard embedding $B^4_\varepsilon\subset B^6_\varepsilon$. 
  Then the map
  $$F:=\Phi_t|_{\Phi_t^{-1}(D^6\setminus B^6_\varepsilon)}\colon D^4\setminus B^4_\varepsilon\approx S^3\times[0,1]\to D^6\setminus B^6_\varepsilon\approx S^5\times[0,1]$$
  is a stable quasi-holomorphic 
 cobordism between the standard embedding and the immersion $f$.
\end{prf}

\begin{rmk}
  Actually, the cobordism constructed above is a stable quasi-holomorphic \textit{concordance}, where for two immersions $f_0,f_1\colon M^3\imto N^5$ this relation is defined by the existence of a stable quasi-holomorphic map $F\colon M^3\times[0,1]\to N^5\times[0,1]$ such that for $i=0,1$ the map $F|_{M^3\times\{i\}}\colon M^3\times\{i\}\to N^5\times\{i\}$ is $f_i$ (i.e. a cobordism through the cylinder over $M^3$). Naturally, the equivalence classes given by quasi-holomorphic homotopy form a finer classification than those given by stable quasi-holomorphic concordance, which, in turn, form a finer classification than those given by stable quasi-holomorphic cobordism. It is an interesting problem to determine how far apart these classifications are from each other. 
\end{rmk}

\begin{rmk}
  In \cite{sphsl} we define \textit{sphere-slicings} of a holomorphic map germ $\Phi\colon(\C^2,0)\to(\C^3,0)$ by fixing a stabilisation $\Phi_t$ and a diffeomorphism $D^6\setminus B^6_\varepsilon\approx S^5\times[0,1]$ (using the notation in the proof of proposition \ref{prop:link}), and considering the restrictions $f_s\colon M_s\to S^5\times\{s\}$ of the map $F$ to the preimages of $S^5\times\{s\}$ under this identification. In other words, we slice the image of a stable deformation of $\Phi$ by a family of $5$-spheres. It turns out that sphere-slicings can be used to identify several invariants of map germs $\Phi$ through the changes of this map when moving between the different slices. Since locally the above constructed concordance is either a quasi-holomorphic homotopy, or an embedded handle attachment, according to remark \ref{rmk:cob}, the local changes of these maps are completely described by \cite{borpow} and the present paper.
\end{rmk}

\section{Local formulae}
\label{sec:form}

\subsection{Moves in a quasi-holomorphic homotopy}
\label{ssec:moves}

The local forms of the Reidemeister type moves appearing along a quasi-holomorphic homotopy are described by the following.

\begin{prop}
  \label{prop:qholformula}
  Let $F\colon M^3\times\R\to N^5\times\R$ be a stable quasi-holomorphic map of the form $(x,s)\mapsto(f_s(x),s)$ for $x\in M^3,s\in\R$, and let $s_0\in\R$ be a parameter such that $f_{s_0}$ has a point of instability $p\in N^5$. Then there is a coordinate neighbourhood $U\subset N^5$ of $p$ such that for a sufficiently small $\varepsilon>0$, for all $s\in[s_0-\varepsilon,s_0+\varepsilon]$ the map $f_s|_{f_s^{-1}(U)}$ is $g_{(s-s_0)/\varepsilon}$ where $g_s,~s\in[-1,1]$ is one of the following:
  \begin{enumerate}[label=\rm{(\roman*)}]
  \item\label{formcc} normal form of the cross-cap move:
    \begin{alignat*}2
      g_s\colon\R^3&\to\R^5,\\
      (x_1,x_2,x_3)&\mapsto(x_1^2-x_2^2+x_3,x_1x_3+2x_1x_2^2-sx_2,x_2x_3-2x_1^2x_2+sx_1,x_3,s-2x_1x_2);
    \end{alignat*}
  \item\label{formtp} normal form of the triple point move: 
    \begin{alignat*}2
      g_s\colon\R^3\sqcup\R^3\sqcup\R^3&\to\R^5,\\
      (x_1,x_2,x_3)&\mapsto(x_1,x_2,x_3,0,0),\\
      (y_1,y_2,y_3)&\mapsto(y_1,s,0,y_2,y_3),\\
      (z_1,z_2,z_3)&\mapsto(0,z_2,z_1,z_2,z_3);
    \end{alignat*}
  \item\label{form0} normal form of the index-$0$ self-tangency move: 
    \begin{alignat*}2
      g_s\colon\R^3\sqcup\R^3&\to\R^5,\\
      (x_1,x_2,x_3)&\mapsto(x_1,x_2,0,x_3,0),\\
      (y_1,y_2,y_3)&\mapsto(y_1,y_2,y_1^2+y_2^2-s,0,y_3);
    \end{alignat*}
  \item\label{form1} normal form of the index-$1$ self-tangency move:
    \begin{alignat*}2
      g_s\colon\R^3\sqcup\R^3&\to\R^5,\\
      (x_1,x_2,x_3)&\mapsto(x_1,x_2,0,x_3,0),\\
      (y_1,y_2,y_3)&\mapsto(y_1,y_2,y_1^2-y_2^2+s,0,y_3);
    \end{alignat*}
  \item\label{form2} normal form of the index-$2$ self-tangency move: 
    \begin{alignat*}2
      g_s\colon\R^3\sqcup\R^3&\to\R^5,\\
      (x_1,x_2,x_3)&\mapsto(x_1,x_2,0,x_3,0),\\
      (y_1,y_2,y_3)&\mapsto(y_1,y_2,y_1^2+y_2^2+s,0,y_3).
    \end{alignat*}
  \end{enumerate}
\end{prop}


Observe that the formulae \ref{form0} and \ref{form2} only differ in the sign of the parameter $s$; in other words, there is only one type of elliptic self-tangency move up to the sign of $s$. The reason for this is that attaching a $2$-handle to the double point surface of the map $F$ is the same as attaching a $0$-handle to it in the opposite direction.

\begin{prf}[of proposition \ref{prop:qholformula}]
  The formulae \ref{formtp}--\ref{form2} were given by Ekholm \cite[proposition 5.3.2]{3knots}. For us it is left to prove that the local form of a cross-cap move is given by the formula in \ref{formcc}. In other words, we have to prove that when slicing the image of the normal form of the cross-cap
  $$h\colon\C^2\to\C^3,(u,v)\mapsto(u^2,uv,v)$$
  by a one-parameter family $\{\R^5_s\}_{s\in\R}$ of affine $5$-spaces in $\C^3$ in a ``generic'' way, we obtain this formula. Here ``generic'' means that if $\R^5_0$ is the space containing the origin in $\C^3$, then $\R^5_0$ is transverse to $h$ in the following sense: at any point sufficiently close to the origin in $\C^2$, the image of the derivative of $h$ (which is a complex plane in $\C^3$) is not contained in $\R^5_0$; and the tangent (complex) line of the double point curve of $c$ at the origin is not contained in $\R^5_0$ either. We will first see that there is a generic such family along which the desired formula appears, then we will see that a generic family is essentially unique.

  By changing local coordinates in the target, we obtain from the map $h$ the map
  $$\tilde h\colon\C^2\to\C^3,(u,v)\mapsto(u^2+v,uv,v),$$
  and by switching to real coordinates via $u=a+ib,v=c+id$ we obtain the map
  $$\tilde g\colon\R^4\to\R^6,(a,b,c,d)\mapsto(a^2-b^2+c,2ab+d,ac-bd,ad+bc,c,d).$$
  The double point set of this map is $D(\tilde g)=\{(a,b,c,d)\in\R^4\mid c=d=0\}$ in the source and $\tilde D(\tilde g)=\{(A,B,C,D,E,F)\in\R^6\mid C=D=E=F=0\}$ in the target. The derivative $d\tilde g_{(a,b,c,d)}$ is
  $$
  \begin{pmatrix}
    2a & -2b & 1 & 0 \\
    2b & 2a & 0 & 1 \\
    c & -d & a & b \\
    d & c & b & a \\
    0 & 0 & 1 & 0 \\
    0 & 0 & 0 & 1
  \end{pmatrix}
  $$
  whose image is everywhere transverse to the hyperplane $\{(A,B,C,D,E,F)\in\R^6\mid B=0\}$, and so are its limits at the origin. This hyperplane is also transverse to $\tilde D(\tilde g)$ at the origin, hence it is an appropriate candidate as the slice $\R^5_0$; indeed, we put
  $$\R^5_s:=\{(A,B,C,D,E,F)\in\R^6\mid B=s\}.$$
  Then the restriction of $\tilde g$ to the preimage of $\R^5_s$ is the map $g_s\colon\R^3\to\R^5$ we get by substituting $x_1=a,x_2=b,x_3=c$ and $s=d+2ab$, that is,
  $$g_s\colon\R^3\to\R^5,(x_1,x_2,x_3)\mapsto(x_1^2-x_2^2+x_3,x_1x_3+2x_1x_2^2-sx_2,x_2x_3-2x_1^2x_2+sx_1,x_3,s-2x_1x_2).$$
  This is the formula we claimed.

  Next we will see that this is essentially the only formula that can appear when moving through a cross-cap by a generic family of real $5$-spaces. To a family $\FF=\{\R^5_s\}_{s\in\R}$ of slicing $5$-spaces in $\R^6$ we can assign the orthogonal complement of the subspace $\R^5_0$. This is a real line $L$ through the origin, i.e. an element in the projective space $\RP^5$. The images of the tangent planes of $\C^2$ under the derivative of the map $h$ outside the origin correspond to elements of $\CP^2$ (via assigning to a plane its orthogonal complement line in $\C^3$). Note that a point $P\in\CP^2$ can be considered as a projective line $\RP^1_P$ in $\RP^5$, and the image of a tangent plane corresponding to $P$ is in $\R^5_0$ precisely if we have $L\in\RP^1_P$. The derivative $dh_{(u,v)}$ is
  $$
  \begin{pmatrix}
    2u & 0 \\
    v & u \\
    0 & v
  \end{pmatrix},
  $$
  hence the orthogonal line of its image is generated by $\big(\frac{v^2}{2u^2},-\frac{v}{u},1\big)$ if $u\ne0$ and by $(1,0,0)$ if $u=0$. In particular, this means that the set of these lines is parametrised by the ratio $\frac vu$ which is a complex number if $u\ne0$ and $\infty$ if $u=0$, that is, an element in $\CP^1$. This means that the family of complex lines orthogonal complement to the image of the derivative of $h$ is a subset of $\CP^2$ homeomorphic to $\CP^1$. 
  Thus the real dimension of the union $X\subset\RP^5$ of the projective lines $\RP^1_P$, where $P$ runs through this family, is $1+2=3$. On the other hand, the orthogonal complement of the tangent (complex) line of the double point curve of $h$ in the target is a real $4$-space which corresponds to a projective subspace $Y\approx\RP^3$ in $\RP^5$. This tangent line is in $\R^5_0$ precisely if we have $L\in Y$.

  Now the family $\FF$ of slicing $5$-spaces in $\R^6$ is generic in the sense of the present proof, if and only if the point $L\in\RP^5$ is not in $X\cup Y$. Two such generic families correspond to two points in $\RP^5\setminus(X\cup Y)$, which can be connected by a path in $\RP^5\setminus(X\cup Y)$ by dimensional reasons. Such a path represents a smooth transformation of one cross-cap move formula to another one, hence the above reasoning yields that up to local diffeotopies there is only one way of representing a cross-cap move in local coordinates.
\end{prf}



\begin{rmk}
  The germs of the cross-cap
  $$\C^2\to\C^3,(u,v)\mapsto(u^2,uv,v)$$
  and the triple point
  $$\C^2\sqcup\C^2\sqcup\C^2\to\C^3,(u_1,v_1)\mapsto(u_1,v_1,0),(u_2,v_2)\mapsto(u_2,0,v_2),(u_3,v_3)\mapsto(0,u_3,v_3)$$
  at the origin are holomorphic (multi)germs with one special point at the origin. They have well-described link immersions (see e.g. \cite{associmm}). The formulae \ref{formcc} and \ref{formtp} in proposition \ref{prop:qholformula} partition these link immersions to three parts: if $D^5$ is the closed ball in $\R^5$ and $S^4$ is its boundary, then the link is partitioned as the union
  $$g_{-1}|_{g_{-1}^{-1}(D^5)}\cup\bigcup_{s\in(-1,1)} g_s|_{g_s^{-1}(S^4)}\cup g_1|_{g_1^{-1}(D^5)}.$$
  The middle part is essentially the the restriction $g_0|_{g_0^{-1}(S^4)}$ extended trivially in an additional dimension; the first and last parts are ``two halves'' of the link immersion glued at the middle part.
\end{rmk}


  

\subsection{Connection to complex Reidemeister moves}
\label{ssec:cplxmove}

In the following we will see that the moves of a quasi-holomorphic homotopy, except for the elliptic self-tangencies, can be partitioned to complex plane curves.

\begin{lemma}
  \label{lemma:cplxmove}
  Let $g_s\colon\coprod^k\R^3\to\R^5,~s\in[-1,1]$ be the normal form of the cross-cap, hyperbolic self-tangency or triple point move (where $k$ is $1$, $2$ or $3$ respectively). Then there is a stable holomorphic plane curve $\gamma\colon\coprod^k\C\imto\C^2$ such that for all $s\ne0$ the map $g_s$ is the union of a real one-parameter family of holomorphic curves $\{\gamma_{r,s}\colon\coprod^k\C\imto\C^2\}_{r\in\R}$ where each $\gamma_{r,s}$ is $\AA$-equivalent to $\gamma$. Further, we have
  \begin{enumerate}[label=\rm{(\roman*)}]
  \item\label{cmovecc} $\gamma(u)=(u-u^3,u^2)$ in the case of the cross-cap move;
  \item\label{cmovest} $\gamma(u)=(u,0)$ and $\gamma(v)=(v,v^2+1)$ in the first and second component respectively in the case of the hyperbolic self-tangency move;
  \item\label{cmovetp} $\gamma(u)=(u,0)$, $\gamma(v)=(1+v,1-v)$ and $\gamma(w)=(0,w)$ in the first, second and third component respectively in the case of the triple point move.
  \end{enumerate}
\end{lemma}

\begin{rmk}
  For the elliptic self-tangency move such a statement cannot hold, since in that case the double point curve of $g_s$ changes as the parameter $s$ passes through $0$, moreover it is compact for all $s$.
\end{rmk}

\begin{rmk}
  \label{rmk:mnb}
  Marar and Nuño-Ballesteros \cite{slicing} defined slicings of holomorphic map germs $(\C^2,0)\to(\C^3,0)$ by complex planes such that at each slice a complex plane curve appears. By stabilising these map germs they obtained that the curves in the slicing planes become stable almost everywhere, and they have \textit{cusps}, \textit{tacnodes} and \textit{triple points} in distinct parameters (these are the complex versions of the three Reidemeister moves). The above formulae \ref{cmovecc}, \ref{cmovest} and \ref{cmovetp} are local forms of the slices around these three instabilities respectively.
\end{rmk}

\begin{prf}[of lemma \ref{lemma:cplxmove}]
  Let $G\colon\coprod^k\R^4\to\R^6$ be the map $(x,s)\mapsto(g_s(x),s)$. In all three cases we will change local coordinates, transforming $G$ to be the union of a (complex) stable one-parameter unfolding $\{\gamma_t\}_{t\in\C}$ of a complex plane curve $\gamma_0$. In this case the maps $\gamma_t$ are pairwise $\AA$-equivalent stable maps for all $t\ne0$.

  In the case of the cross-cap move the map $G$ has the form
  \begin{alignat*}2
    G\colon\R^4&\to\R^6,\\
    (x_1,x_2,x_3,s)&\mapsto(x_1^2-x_2^2+x_3,x_1x_3+2x_1x_2^2-sx_2,x_2x_3-2x_1^2x_2+sx_1,x_3,s-2x_1x_2,s).
  \end{alignat*}
  By the local coordinate changes $(x_1,x_2,x_3,s)\mapsto(x_1,x_2,x_3-x_1^2+x_2^2,s)$ in the source and $(A,B,C,D,$ $E,F)\mapsto(B,C,-D+A,-E+B,A,F)$ in the target we obtain the map
  $$(x_1,x_2,x_3,s)\mapsto(x_1x_3-x_1^3+3x_1x_2^2-sx_2,x_2x_3+x_2^3-3x_1^2x_2+sx_1,x_1^2-x_2^2,2x_1x_2,x_3,s).$$
  This, when switching to the complex coordinates $u=x_1+ix_2,t=x_3+is$, has the form
  $$(u,t)\mapsto(ut-u^3,u^2,t).$$
  This defines the maps
  $$\gamma_t\colon\C\to\C^2,u\mapsto(ut-u^3,u^2),~t\in\C,$$
  which form a stable one-parameter unfolding of the plane curve $\gamma_0\colon u\mapsto(-u^3,u^2)$. For $r,s\in\R$ we define $\gamma_{r,s}:=\gamma_{r+is}$ and $\gamma:=\gamma_1$ which proves \ref{cmovecc}.

  Turning to the case of the hyperbolic self-tangency move, we have
  \begin{alignat*}2
    G\colon\R^4\sqcup\R^4&\to\R^6,\\
    (x_1,x_2,x_3,s)&\mapsto(x_1,x_2,0,x_3,0,s),\\
    (y_1,y_2,y_3,s)&\mapsto(y_1,y_2,y_1^2-y_2^2+s,0,y_3,s).
  \end{alignat*}
  By the local coordinate changes $(x_1,x_2,x_3,s)\mapsto(x_1,x_2,x_3+2x_1x_2,s)$ and $(y_1,y_2,y_3,s)\mapsto(y_1,y_2,$ $y_3+2y_1y_2,s)$ in the source and $(A,B,C,D,E,F)\mapsto(A,B,C,E,D-2AB+E,F)$ in the target we obtain the map
  \begin{alignat*}2
    (x_1,x_2,x_3,s)&\mapsto(x_1,x_2,0,0,x_3,s),\\
    (y_1,y_2,y_3,s)&\mapsto(y_1,y_2,y_1^2-y_2^2+s,y_3+2y_1y_2,y_3,s).
  \end{alignat*}
  This, when switching to the complex coordinates $u=x_1+ix_2,t=s+ix_3$ in the first component and $v=y_1+iy_2,t=s+iy_3$ in the second component, has the form
  $$(u,t)\mapsto(u,0,t),(v,t)\mapsto(v,v^2+t,t).$$
  This defines the maps
  $$\gamma_t\colon\C\sqcup\C\to\C^2,u\mapsto(u,0),v\mapsto(v,v^2+t),~t\in\C,$$
  which, again, form a stable one-parameter unfolding of the plane curve $\gamma_0\colon u\mapsto(u,0),v\mapsto(v,v^2)$. Then for $r,s\in\R$ we define $\gamma_{r,s}:=\gamma_{s+ir}$ and $\gamma:=\gamma_1$ proving \ref{cmovest}.

  Lastly, in the case of the triple point move we have
  \begin{alignat*}2
    G\colon\R^4\sqcup\R^4\sqcup\R^4&\to\R^6,\\
    (x_1,x_2,x_3,s)&\mapsto(x_1,x_2,x_3,0,0,s),\\
    (y_1,y_2,y_3,s)&\mapsto(y_1,s,0,y_2,y_3,s),\\
    (z_1,z_2,z_3,s)&\mapsto(0,z_2,z_1,z_2,z_3,s).
  \end{alignat*}
  By the local coordinate changes $(x_1,x_2,x_3,s)\mapsto(x_1,\frac12x_2,x_3-\frac12x_1,s)$, $(y_1,y_2,y_3,s)\mapsto(y_1+y_3,\frac12s-\frac12y_2,y_3-y_1,s)$ and $(z_1,z_2,z_3,s)\mapsto(z_3-\frac12z_1,\frac12z_2,z_1,s)$ in the source and $(A,B,C,D,E,F)\mapsto(A,2B-2D,E,2D,C+\frac12A+\frac12E,F)$ in the target we obtain the map
  \begin{alignat*}2
    (x_1,x_2,x_3,s)&\mapsto(x_1,x_2,0,0,x_3,s),\\
    (y_1,y_2,y_3,s)&\mapsto(y_1+y_3,y_2+s,y_3-y_1,s-y_2,y_3,s),\\
    (z_1,z_2,z_3,s)&\mapsto(0,0,z_1,z_2,z_3,s).
  \end{alignat*}
  This, when switching to the complex coordinates $u=x_1+ix_2,t=x_3+is$ in the first, $v=y_1+iy_2,t=y_3+is$ in the second and $w=z_1+iz_2,t=z_3+is$ in the third component, has the form
  $$(u,t)\mapsto(u,0,t),(v,t)\mapsto(t+v,t-v,t),(w,t)\mapsto(0,w,t).$$
  This defines the maps
  $$\gamma_t\colon\C\sqcup\C\sqcup\C\to\C^2,u\mapsto(u,0),v\mapsto(t+v,t-v),w\mapsto(0,w),~t\in\C,$$
  which, yet again, form a stable one-parameter unfolding of the plane curve $\gamma_0\colon u\mapsto(u,0),v\mapsto(v,-v),w\mapsto(0,w)$. Then for $r,s\in\R$ we define $\gamma_{r,s}:=\gamma_{r+is}$ and $\gamma:=\gamma_1$ proving \ref{cmovetp}. This completes the proof in all cases.
\end{prf}

\begin{rmk}
  \label{rmk:gamma}
  In the above proof the curve $\gamma$ was defined as $\gamma_{1,0}$ for the cross-cap and triple point moves, however, it is $\gamma_{0,1}$ for the hyperbolic self-tangency move. Hence in the cases \ref{cmovecc} and \ref{cmovetp} the partition of the map $g_s$ (for $s\ne0$) to the family of curves $\{\gamma_{r,s}\}_{r\in\R}$ (where each $\gamma_{r,s}$ is identified with $\gamma$) is independent of the parameter $s$, however, in the case \ref{cmovest} it depends on $s$. In particular, if for $s=-1$ we apply the local coordinate change $(U,V)\mapsto(U,V-U^2+1)$ in the target $\C^2$, then we obtain that the partition of the map $g_{-1}$ to the family $\{\gamma_{r,-1}\}_{r\in\R}$ corresponds to the partition of the map $g_1$ to the family $\{\gamma_{r,1}\}_{r\in\R}$ by exchanging the two source components.
\end{rmk}

\section{Fundamental group of the complement}
\label{sec:pi1}

\subsection{Knot groups of immersions}
\label{ssec:knotgrp}

The fundamental group of the complement of the image of a $2$-codimensional embedding is called its \textit{knot group} and it is an important invariant of its isotopy class; see \cite{kerv} for spheres and \cite{smith} for general manifolds (possibly with boundaries). If we change ``embedding'' to ``stable immersion'' and ``isotopy'' to ``homotopy through stable immersions'', then the same fundamental group becomes a similarly key invariant, which we will also call the \textit{knot group}. Proposition \ref{prop:isot} and remark \ref{rmk:isot} immediately imply the following.

\begin{crly}
  For two manifolds $M^n$ and $N^{n+2}$, the knot group is constant on the path components of the space of stable immersions $M^n\imto N^{n+2}$ and extends the notion of knot group from the space of embeddings $M^n\into N^{n+2}$.
\end{crly}

\begin{rmk}
  \label{rmk:immemb}
  For a stable immersion $f\colon M^n\imto N^{n+2}$, removing small tubular neighbourhoods of the multiple points of $f$ both in the source and in the target yields an embedding $\widehat f\colon\widehat M^n\into\widehat N^{n+2}$ between manifolds with corners (homeomorphic to manifolds with boundaries). Clearly, $N^{n+2}\setminus f(M^n)$ is homeomorphic to $\widehat N^{n+2}\setminus\widehat f(\widehat M^n)$, hence the immersion $f$ and the embedding $\widehat f$ have the same knot group. 
\end{rmk}

The following notion was defined for $2$-codimensional embeddings in \cite{smith}; we extend it to immersions now.

\begin{defi}
  For a stable immersion $f\colon M^n\imto N^{n+2}$, an element in $\pi_1(N^{n+2}\setminus f(M^n))$ is called a \textit{meridian class}, if it is represented by the embedding of a fibre of the circle bundle of the normal bundle of $f$. The kernel of the homomorphism $\pi_1(N^{n+2}\setminus f(M^n))\to\pi_1(N^{n+2})$ is called the \textit{meridian subgroup}.
\end{defi}

\begin{rmk}
  \label{rmk:meridgen}
  By remark \ref{rmk:immemb} and \cite[proposition I.3]{smith} the meridian subgroup of the knot group is 
  the normal closure of one meridian class.
\end{rmk}

\begin{rmk}
  \label{rmk:meridbound}
  If $f\colon M^n\imto N^{n+2}$ is an immersion, and $W\subset N^{n+2}$ is a tubular neighbourhood of the image of $f$, then the complement of $W$ is a manifold with corners, hence the boundary $\partial W$ is the union of the smooth boundary part $\overline\partial W$ and the corners $\widehat\partial W$. Meridian classes in the knot group of $f$ are elements in the image of the homomorphism induced by the inclusion $\overline\partial W\subset N^{n+2}\setminus f(M^n)$ in fundamental groups.
\end{rmk}

Although knot groups of $2$-codimensional embeddings (thus also of immersions) were extensively analysed (e.g. in the above cited papers), to our knowledge the change of a knot group under a homotopy through an instability was not studied. Theorems \ref{thm:noell} and \ref{thm:cyc} describe the change in the knot groups of immersions $M^3\imto N^5$ when we allow unstable immersions and maps with isolated singularities corresponding to cross-cap moves.

\subsection{Proof of the main theorems}
\label{ssec:prfmain}



Recall the main theorems:

\begin{thmref}{thm:noell}
  If $f_s\colon M^3\to N^5,~s\in[-1,1]$ is a quasi-holomorphic homotopy without elliptic self-tangency moves, then the groups $\pi_1(N^5\setminus f_{-1}(M^3))$ and $\pi_1(N^5\setminus f_1(M^3))$ are isomorphic.
\end{thmref}

\begin{thmref}{thm:cyc}
  If $f_s\colon M^3\to N^5,~s\in[-1,1]$ is a quasi-holomorphic homotopy which has only one elliptic self-tangency move, then there are presentations of the groups $\pi_1(N^5\setminus f_{\pm1}(M^3))$ such that their generator sets coincide and their relation sets only differ in one commutator relation: for two generators $a,b$ one of the above groups possesses the relation $ab=ba$ and the other one does not.
\end{thmref}

Their proofs will be started together, then we separate the conditions later. Clearly it is sufficient to 
describe how the knot group changes when passing through one instability. Thus we may suppose that $f_s\colon M^3\to N^5,~s\in[-1,1]$ is a quasi-holomorphic homotopy where $f_s$ is a stable immersion for all $s\ne0$ and the map $f_0$ has only one point of instability $p\in N^5$.  
We will denote the image $f_s(M^3)$ by $M_s\subset N^5$. Let $F\colon M^3\times[-1,1]\to N^5\times[-1,1]$ be the map $F(x,s):=(f_s(x),s)$. By possibly changing $f_{\pm1}$ to $f_{\pm\varepsilon}$ for some small $\varepsilon>0$, we can assume that there is a coordinate neighbourhood $U\subset N^5$ of $p$ such that inside $U\times[-1,1]$ the map $F$ has one of the forms described in proposition \ref{prop:qholformula} 
and in $(N^5\setminus U)\times[-1,1]$ it is of the form $f_0|_{f_0^{-1}(N^5\setminus U)}\times\id_{[-1,1]}$. Note that the subspace $N^5\setminus(U\cup M_s)$ is the same for all $s\in[-1,1]$.

Let $i^\pm\colon\partial U\setminus M_{\pm1}\into N^5\setminus(U\cup M_{\pm1})$ and $j^\pm\colon\partial U\setminus M_{\pm1}\into\ol U\setminus M_{\pm1}$ be the inclusions. By the Van Kampen theorem we have
$$\pi_1(N^5\setminus M_{\pm1})\cong\Big(\pi_1(N^5\setminus(U\cup M_{\pm1}))*\pi_1(\ol U\setminus M_{\pm1})\Big)/R_\pm$$
where $*$ denotes free product and $R_\pm$ is the normal subgroup generated by the elements $i^\pm_\#(a)j^\pm_\#(a)^{-1}$ for all $a\in\pi_1(\partial U\setminus M_{\pm1})$. Note that we have natural isomorphisms $\pi_1(\partial U\setminus M_{-1})\cong\pi_1(\partial U\setminus M_1)$ and $\pi_1(N^5\setminus(U\cup M_{-1}))\cong\pi_1(N^5\setminus(U\cup M_1))$, and the homomorphisms $i^-_\#$ and $i^+_\#$ coincide. Hence the only possible change in the fundamental group happens in
$$j^\pm_\#\colon\pi_1(\partial U\setminus M_{\pm1})\to\pi_1(\ol U\setminus M_{\pm1}).$$
Put $G:=\pi_1(N^5\setminus(U\cup M_{\pm1}))$. From now on we separate the two theorems.

\begin{prf}[of theorem \ref{thm:noell}]
  Let $\gamma\colon\coprod^k\C\imto\C^2$ be the complex plane curve in lemma \ref{lemma:cplxmove} (where $k$ is $1$, $2$ or $3$ depending on the type of move); let $D^4\subset\C^2$ be the closed ball of a sufficiently large radius, and $S^3:=\partial D^4$. Then the space $\ol U\setminus M_{\pm1}$ is homeomorphic to the space $(D^4\setminus\gamma(\gamma^{-1}(D^4)))\times[0,1]$, and $\partial U\setminus M_{\pm1}$ to its boundary, which is the union of the space $(S^3\setminus\gamma(\gamma^{-1}(S^3)))\times[0,1]$ and two copies of the space $D^4\setminus\gamma(\gamma^{-1}(D^4))$ glued to it on the two boundary components.
  
  Suppose first that the instability of the map $f_0$ corresponds to a cross-cap move. Then $\gamma\colon\C\to\C^2$ is the map $u\mapsto(u-u^3,u^2)$. In real coordinates this map has the form
  $$(x_1,x_2)\mapsto(x_1-x_1^3+3x_1x_2^2,x_2+x_2^3-3x_1^2x_2,x_1^2-x_2^2,2x_1x_2),$$
  which has one double point $(0,0,1,0)\in\R^4$. Let $X$ and $Y$ be the image of this map intersected by the half-spaces $H_1:=\{(A,B,C,D)\in\R^4\mid C<\frac12\}$ and $H_2:=\{(A,B,C,D)\in\R^4\mid C>\frac12\}$ respectively. Then the embeddings $X\into H_1$ and $Y\into H_2$ are isotopic to the standard embedding $X\approx\R^2\subset\R^4\approx H_1$ and the embedding of two orthogonal complement planes $Y\approx\R^2\vee\R^2\subset\R^4\approx H_2$ respectively; let $Y_1\subset Y$ and $Y_2\subset Y$ be the subspaces corresponding to these $\R^2$'s. In the intersection $\ol H_1\cap\ol H_2$ we get that $\partial X=\partial Y$ is isotopic to the embedding of two disjoint lines; moreover $\partial X\subset\ol X\approx D^2$ is isotopic to the embedding of two open intervals into the boundary of $D^2$ and $\partial Y\subset\ol Y$ is isotopic to the embedding of one open interval into the boundary of $\ol Y_1\approx D^2$ and one into the boundary of $\ol Y_2\approx D^2$. The fundamental group of $\ol H_1\setminus\ol X\cong S^1$ is $\Z$ generated by the homotopy class of a fibre of the circle bundle of the normal bundle of $X\subset H_1$; the fundamental group of $\ol H_2\setminus\ol Y\cong T^2$ is $\Z^2$ generated by the homotopy classes of fibres of the circle bundles of the normal bundles of $Y_1\subset H_2$ and $Y_2\subset H_2$. Thus the Van Kampen theorem implies that we have
  $$\pi_1(D^4\setminus\gamma(\gamma^{-1}(D^4)))\cong\Z$$
  generated by the homotopy class of a fibre of the circle bundle of the normal bundle of the immersion $\gamma$. This yields (again using the Van Kampen theorem) that both $\pi_1(\partial U\setminus M_{\pm1})$ and $\pi_1(\ol U\setminus M_{\pm 1})$ are isomorphic to $\Z$ generated by $a_\pm$ and $a'_\pm$ respectively, and $j^\pm_\#$ is the isomorphism $a_\pm\mapsto a'_\pm$. Since the natural identification of $\partial U\setminus M_{-1}$ and $\partial U\setminus M_1$ takes the generator $a_-$ to $a_+$ (see remark \ref{rmk:gamma}), we obtain that $i^\pm_\#(a_\pm)$ is the same element in $G$. Thus we get
  \begin{alignat*}2
    \pi_1(N^5\setminus M_{-1})&\cong\Big(G*\Z\la a'_-\ra\Big)/(i^-_\#(a_-){a'_-}^{-1})\cong \\
                              &\cong\Big(G*\Z\la a'_+\ra\Big)/(i^+_\#(a_+){a'_+}^{-1})\cong\pi_1(N^5\setminus M_1).
  \end{alignat*}
  
  Now suppose that the instability of the map $f_0$ corresponds to a hyperbolic self-tangency move. Then $\gamma\colon\C\sqcup\C\to\C^2$ is the map $u\mapsto(u,0),v\mapsto(v,v^2+1)$. In real coordinates this map has the form
  \begin{alignat*}2
    (x_1,x_2)&\mapsto(x_1,x_2,0,0),\\
    (y_1,y_2)&\mapsto(y_1,y_2,y_1^2-y_2^2+1,2y_1y_2),
  \end{alignat*}
  which has two double points $(0,\pm1,0,0)\in\R^4$. Let $X$ and $Y$ be the image of this map intersected by the half-spaces $H_1:=\{(A,B,C,D)\in\R^4\mid B<0\}$ and $H_2:=\{(A,B,C,D)\in\R^4\mid B>0\}$ respectively. Then the embeddings $X\into H_1$ and $Y\into H_2$ are both isotopic to embeddings of two orthogonal complement planes $X\approx Y\approx\R^2\vee\R^2\subset\R^4\approx H_1\approx H_2$; let $X_1,X_2\subset X$ and $Y_1,Y_2\subset Y$ be the subspaces corresponding to these $\R^2$'s. The fundamental groups of $\ol H_1\setminus\ol X\cong T^2$ and $\ol H_2\setminus\ol Y\cong T^2$ are both $\Z^2$ generated by the homotopy classes of fibres of the circle bundles of the normal bundles of $X_1$ and $X_2$ and of $Y_1$ and $Y_2$ respectively. Similarly to the case of the cross-cap move, here the Van Kampen theorem yields that in the fundamental group of their union the generators corresponding to $X_1$ and $Y_1$ are identified and so are the generators corresponding to $X_2$ and $Y_2$. Hence we have
  $$\pi_1(D^4\setminus\gamma(\gamma^{-1}(D^4)))\cong\Z^2$$
  generated by the homotopy classes of fibres of the circle bundles of the normal bundles of the restriction of $\gamma$ to the two components. This yields (again using the Van Kampen theorem) that both $\pi_1(\partial U\setminus M_{\pm1})$ and $\pi_1(\ol U\setminus M_{\pm 1})$ are isomorphic to $\Z^2$ generated by $a_\pm,b_\pm$ and $a'_\pm,b'_\pm$ respectively, and $j^\pm_\#$ is the isomorphism $a_\pm\mapsto a'_\pm,b_\pm\mapsto b'_\pm$. Now the natural identification of $\partial U\setminus M_{-1}$ and $\partial U\setminus M_1$ takes the generator $a_-$ to $b_+$ and $b_-$ to $a_+$ (see remark \ref{rmk:gamma}), hence we have $i^+_\#(a_+)=i^-_\#(b_-)$ and $i^+_\#(b_+)=i^-_\#(a_-)$ in $G$. Thus we get
  \begin{alignat*}2
    \pi_1(N^5\setminus M_{-1})&\cong\Big(G*\Z\la a'_-,b'_-\ra\Big)/(i^-_\#(a_-){a'_-}^{-1},i^-_\#(b_-){b'_-}^{-1})\cong \\
                              &\cong\Big(G*\Z\la a'_+,b'_+\ra\Big)/(i^+_\#(b_+){a'_+}^{-1},i^+_\#(a_+){b'_+}^{-1})\cong\pi_1(N^5\setminus M_1).
  \end{alignat*}
  
  Lastly, suppose that the instability of the map $f_0$ corresponds to a triple point move. Then $\gamma\colon\C\sqcup\C\sqcup\C^2$ is the map $u\mapsto(u,0),v\mapsto(1+v,1-v),w\mapsto(0,w)$. In real coordinates this map has the form
  \begin{alignat*}2
    (x_1,x_2)&\mapsto(x_1,x_2,0,0),\\
    (y_1,y_2)&\mapsto(1+y_1,y_2,1-y_1,-y_2),\\
    (z_1,z_2)&\mapsto(0,0,z_1,z_2),
  \end{alignat*}
  which has three double points $(0,0,0,0),(2,0,0,0),(0,0,2,0)\in\R^4$. Let $X$, $Y$ and $Z$ be the image of this map intersected by $H_1:=\{(A,B,C,D)\in\R^4\mid A,C<1\}$, $H_2:=\{(A,B,C,D)\in\R^4\mid A>1\}$ and $H_2:=\{(A,B,C,D)\in\R^4\mid A<1,C>1\}$ respectively. Then all three of the embeddings $X\into H_1$, $Y\into H_2$ and $Z\into H_3$ are isotopic to embeddings of two orthogonal complement planes $X\approx Y\approx Z\approx\R^2\vee\R^2\subset\R^4\approx H_1\approx H_2\approx H_3$; let $X_1,X_2\subset X$, $Y_1,Y_2\subset Y$ and $Z_1,Z_2\subset Z$ be the subspaces corresponding to these $\R^2$'s. The fundamental groups of $\ol H_1\setminus\ol X\cong T^2$, $\ol H_2\setminus\ol Y\cong T^2$ and $\ol H_3\setminus\ol Z\cong T^2$ are all $\Z^2$ generated by the homotopy classes of fibres of the circle bundles of the normal bundles of $X_1$ and $X_2$, of $Y_1$ and $Y_2$ and of $Z_1$ and $Z_2$ respectively. Yet again the Van Kampen theorem implies that in the fundamental group $\pi_1((\ol H_1\setminus\ol X)\cup(\ol H_2\setminus\ol Y))$ the generators corresponding to $X_1$ and $Y_1$ are identified, then in the fundamental group of the union of this with $\ol H_3\setminus\ol Z$ the generators corresponding to $Z_1$ and $Z_2$ are identified with those corresponding to $X_2$ and $Y_2$ respectively. Hence we have
  $$\pi_1(D^4\setminus\gamma(\gamma^{-1}(D^4)))\cong\Z^3$$
  generated by the homotopy classes of fibres of the circle bundles of the normal bundles of the restriction of $\gamma$ to the three components. This yields (using the Van Kampen theorem once more) that both $\pi_1(\partial U\setminus M_{\pm1})$ and $\pi_1(\ol U\setminus M_{\pm1})$ are isomorphic to $\Z^3$ generated by $a_\pm,b_\pm,c_\pm$ and $a'_\pm,b'_\pm,c'_\pm$ respectively, and $j^\pm_\#$ is the isomorphism $a_\pm\mapsto a'_\pm,b_\pm\mapsto b'_\pm,c_\pm\mapsto c'_\pm$. Now, as in the case of the cross-cap move, the natural identification of $\partial U\setminus M_{-1}$ and $\partial U\setminus M_1$ takes the generators $a_-,b_-,c_-$ to $a_+,b_+,c_+$ respectively (see remark \ref{rmk:gamma}), hence we have $i^+_\#(a_+)=i^-_\#(a_-)$, $i^+_\#(b_+)=i^-_\#(b_-)$ and $i^+_\#(c_+)=i^-_\#(c_-)$ in $G$. Thus we get
  \begin{alignat*}2
    \pi_1(N^5\setminus M_{-1})&\cong\Big(G*\Z\la a'_-,b'_-,c'_-\ra\Big)/(i^-_\#(a_-){a'_-}^{-1},i^-_\#(b_-){b'_-}^{-1},i^-_\#(c_-){c'_-}^{-1})\cong \\
                              &\cong\Big(G*\Z\la a'_+,b'_+,c'_+\ra\Big)/(i^+_\#(a_+){a'_+}^{-1},i^+_\#(b_+){b'_+}^{-1},i^+_\#(c_+){c'_+}^{-1})\cong\pi_1(N^5\setminus M_1).
  \end{alignat*}

  This concludes the proof in the case of all three moves.
\end{prf}

\begin{rmk}
  \label{rmk:monodromy}
  Consider the complex plane curves $\gamma_{r,s},~(r,s)\in\R^2$ in lemma \ref{lemma:cplxmove}, which were essentially used in the above proof. One would expect that the monodromy action of the parameters traversing the circle $S^1\subset\R^2$ appears in the change of fundamental groups computed above; however, it does not. In the computation of the knot group of $f_{-1}$ (resp. $f_1$) we use the family of curves $\{\gamma_{r,-1}\}_{r\in\R}$ (resp. $\{\gamma_{r,1}\}_{r\in\R}$). In all three cases we found that the embedding of one ``fibre'' $D^4\setminus\gamma_{r,\pm1}(\gamma_{r,\pm1}^{-1}(D^4))$ (for some fixed $r$) into both $\ol U\setminus M_{\pm1}$ and $\partial U\setminus M_{\pm1}$ induces isomorphisms in the fundamental groups. Hence in the application of the Van Kampen theorem, when we identify the gluing regions $\partial U\setminus M_{-1}$ and $\partial U\setminus M_{1}$ with each other, all information on the fundamental group is contained in the ``fibres'' $D^4\setminus\gamma_{r,-1}(\gamma_{r,-1}^{-1}(D^4))$ and $D^4\setminus\gamma_{r,1}(\gamma_{r,1}^{-1}(D^4))$ which are identified via the second parameter $s$ traversing $[-1,1]$. This identification of ``fibres'' can depend on whether the fixed parameter $r$ is positive or negative (this is where the monodromy appears), but fixing any choice yields an isomorphism of fundamental groups. Consequently, the knot groups of $f_{\pm1}$ are isomorphic, but the monodromy action can come up in the concrete construction of their isomorphism.
\end{rmk}
  
\begin{prf}[of theorem \ref{thm:cyc}]
  We can suppose that the instability of the map $f_0$ corresponds to an index-$0$ self-tangency move (see proposition \ref{prop:qholformula}). The map $f_{\pm1}$ restricted to the preimage of $\ol U$ is of the form
  \begin{alignat*}2
    g_{\pm1}\colon D^3\sqcup D^3&\to D^5,\\
    (x_1,x_2,x_3)&\mapsto(x_1,x_2,0,x_3,0),\\
    (y_1,y_2,y_3)&\mapsto(y_1,y_2,y_1^2+y_2^2\mp1,0,y_3),
  \end{alignat*}
  where $D^3$ and $D^5$ are closed balls in $\R^3$ and $\R^5$ respectively of sufficiently large radius. This implies that $g_{-1}$ can be identified with the embedding of two parallel subspaces $\R^3$ into $\R^5$ (intersected by $D^5$). Hence the groups $\pi_1(\partial U\setminus M_{-1})$ and $\pi_1(\ol U\setminus M_{-1})$ are both isomorphic to the free group on two elements generated by the homotopy classes of fibres of the circle bundles of the normal bundle of $f_{-1}$ restricted to the two components of the preimage of $\partial U$ and $\ol U$ respectively. Denote these generators by $a_-,b_-$ and $a'_-,b'_-$ respectively; then $j^-_\#$ is the isomorphism $a_-\mapsto a'_-,b_-\mapsto b'_-$. Moving to $f_1$, the fundamental group $\pi_1(\partial U\setminus M_1)$ is still the free group generated by the homotopy classes $a_+$ and $b_+$ of fibres of the circle bundles of the normal bundle of $f_1|_{f_1^{-1}(\partial U)}$ restricted to the two components. The natural identification of $\partial U\setminus M_{-1}$ and $\partial U\setminus M_1$ takes the generators $a_-,b_-$ to $a_+,b_+$ respectively, hence $i^-_\#(a_-)=i^+_\#(a_+),i^-_\#(b_-)=i^+_\#(b_+)$.

  To compute the fundamental group of $\ol U\setminus M_1$ observe that the map $g_1$ has a rotational symmetry with respect to the $2$-codimensional subspace $H:=\{(A,B,C,D,E)\in D^5\mid A=B=0\}$. One fibre of this symmetry is the restriction of $g_1$ to the preimage of $F:=\{(A,B,C,D,E)\in D^5\mid A\ge0,B=0\}$. Let $X$ denote the image of $g_1$ intersected by $F$. Then the embedding $X\into F$ is isotopic to the embedding of two orthogonal subspaces $\R^2\vee\R^2\subset\R^4$ intersected by a closed ball; let $X_1,X_2\subset X$ be the subspaces corresponding to these $\R^2$'s. Then for $i=1,2$ we have $X_i\approx D^2$ and the intersection $X_i\cap H$ is an interval embedded into the boundary of this disk. Then we have
  $$\ol U\setminus M_1\approx\Big((F\setminus X)\times S^1\Big)/\big((x,t_1)\sim(x,t_2),~x\in H,t_1,t_2\in S^1\big).$$
  The space $F\setminus X$ deformation retracts to the torus $T^2$ and this retraction takes the subspace $H\setminus X$ to the union of a longitude and a meridian $S^1\vee S^1\subset T^2$. Hence we have
  $$\ol U\setminus M_1\cong\big(T^2\times S^1\big)/\big((x,t_1)\sim(x,t_2),~x\in S^1\vee S^1,t_1,t_2\in S^1\big),$$
  which is homotopy equivalent to the union of the spaces $T^2\times S^1$ and $(S^1\vee S^1)\times D^2$ glued along $(S^1\vee S^1)\times S^1$. Thus by the Van Kampen theorem the fundamental group of $\ol U\setminus M_1$ is $\Z^2$ generated by the homotopy classes of fibres of the circle bundles of the normal bundle of the restriction of $f_1$ to the two components of the preimage of $\ol U$.

  Denoting the above generators of $\pi_1(\ol U\setminus M_1)$ by $a'_+,b'_+$, we get $j^+_\#(a_+)=a'_+$ and $j^+_\#(b_+)=b'_+$. Hence we have
  \begin{alignat*}2
    \pi_1(N^5\setminus M_{-1})&\cong\Big(G*\la a'_-,b'_-\ra\Big)/(i^+_\#(a_+){a'_-}^{-1},i^+_\#(b_+){b'_-}^{-1}),\\
    \pi_1(N^5\setminus M_1)&\cong\Big(G*\Z\la a'_+,b'_+\ra\Big)/(i^+_\#(a_+){a'_+}^{-1},i^+_\#(b_+){b'_+}^{-1}).
  \end{alignat*}
  Thus there are presentations of the groups $\pi_1(N^5\setminus M_{\pm1})$ such that their generator and relation sets are those of $G$, with two additional generators $a,b$ (corresponding to $a'_\pm,b'_\pm$ respectively) and additional relations corresponding to the identification of $a,b$ with $i^+_\#(a_+),i^+_\#(b_+)$ respectively, and in the case of $\pi_1(N^5\setminus M_1)$ yet another additional relation $ab=ba$. This is what we claimed.
\end{prf}

\begin{rmk}
  \label{rmk:meridcyc}
  The above proof shows that at least one of the generators $a,b$ in theorem \ref{thm:cyc} is a meridian class in both groups $\pi_1(N^5\setminus f_{\pm1}(M^3))$. 
\end{rmk}

\begin{rmk}
  \label{rmk:index}
  The above proof also shows that way the new relation $ab=ba$ in theorem \ref{thm:cyc} appears, depends on the index of the handle attached to the double point surface in the self-tangency move: if it is a $0$-handle (resp. $2$-handle) attachment when moving from $-1$ to $1$, then the relation $ab=ba$ is present in the group $\pi_1(N^5\setminus f_{1}(M^3))$ (resp. in $\pi_1(N^5\setminus f_{-1}(M^3))$).
\end{rmk}

\subsection{Corollaries and final observations}
\label{ssec:crly}

The following is immediate from theorems \ref{thm:noell} and \ref{thm:cyc}.

\begin{crly}
  If $f_0,f_1\colon M^3\imto N^5$ are quasi-holomorphically homotopic immersions, then their knot groups have presentations on the same generator set such that their relation sets only differ in some commutators. In particular, we have
  $$H_1(N^5\setminus f_0(M^3))\cong H_1(N^5\setminus f_1(M^3)).$$
\end{crly}

Taking remark \ref{rmk:meridcyc} into account we get:

\begin{crly}
  \label{crly:meridfactor}
  If $f_0,f_1\colon M^3\imto N^5$ are quasi-holomorphically homotopic immersions, 
  then the factors of their knot groups by their meridian subgroups are isomorphic.
\end{crly}

Remarks \ref{rmk:meridcyc} and \ref{rmk:meridgen} together also imply:

\begin{crly}
  If $f_0,f_1\colon M^3\imto N^5$ are quasi-holomorphically homotopic immersions, and for both $f_0$ and $f_1$ there is a meridian class in the centre of the knot group (or, equivalently, the meridian subgroup is in the centre), then the knot groups of $f_0$ and $f_1$ are isomorphic.
\end{crly}




The combination of remarks \ref{rmk:meridbound} and \ref{rmk:index} yields:

\begin{crly}
  \label{crly:closeloops}
  Let $f_0,f_1\colon M^3\imto N^5$ be stable immersions, let $W_0,W_1\subset N^5$ be tubular neighbourhoods of the images of $f_0,f_1$ respectively, and denote for $i=0,1$ by $j_i$ 
  the embedding of the boundary $\partial W_i\into N^5\setminus f_i(M^3)$. 
  Suppose that the induced homomorphism ${j_0}_\#$ 
  in fundamental groups is surjective, i.e. all elements in the knot group of $f_0$ are represented by loops in $W_0$. If there is a quasi-holomorphic homotopy $f_s\colon M^3\to N^5,~s\in[0,1]$ whose every elliptic self-tangency move corresponds to a $0$-handle attachment when moving from $0$ to $1$, then the induced homomorphism ${j_1}_\#$ 
  is surjective as well, i.e. all elements in the knot group of $f_1$ are represented by loops in $W_1$.
\end{crly}


  

Theorems \ref{thm:noell} and \ref{thm:cyc}, together with remarks \ref{rmk:meridcyc} and \ref{rmk:index} also imply:

\begin{crly}
  \label{crly:no2h}
  If $f_0,f_1\colon M^3\imto N^5$ are stable immersions such that 
  the knot group of $f_0$ is commutative, and there is a quasi-holomorphic homotopy $f_s\colon M^3\to N^5,~s\in[0,1]$ whose every elliptic self-tangency move corresponds to a $0$-handle attachment when moving from $0$ to $1$, then the knot groups of $f_0$ and $f_1$ are isomorphic.
\end{crly}

\begin{ex}
  \label{ex:sphno2h}
  The standard embedding $S^3\subset S^5$ is such that its knot group is $\Z$, generated by a meridian class. 
  Thus corollary \ref{crly:no2h} implies that any immersion $f\colon S^3\imto S^5$, which is quasi-holomorphically homotopic without index-$2$ self-tangency moves to the standard embedding, has the same knot group. In other words such immersions are ``simple'' in the same sense as defined for usual knots. In particular, this holds if the quasi-holomorphic homotopy has only one elliptic self-tangency move, since that necessarily corresponds to a $0$-handle attachment.
\end{ex}










The above example could be interesting because of applications on link immersions of map germs. In \cite{sphsl} we show that the quasi-holomorphic concordance between the standard embedding and the link immersion constructed in the proof of proposition \ref{prop:link} can be chosen in such a way that it has only one elliptic self-tangency move. 
Based on this, we ask the following.

\begin{ques}
  \label{ques:link}
  Let $f\colon S^3\imto S^5$ be the link immersion of a finitely $\AA$-determined holomorphic map germ $(\C^2,0)\to(\C^3,0)$. Is $f$ quasi-holomorphically homotopic to the standard embedding $S^3\subset S^5$? If so, can this homotopy be chosen such that its elliptic self-tangency moves all correspond to $0$-handle attachments?
\end{ques}



If the construction in the proof of proposition \ref{prop:link} could be altered in such a way that the concordance becomes a homotopy (i.e. if the complementary handles could be removed), that would give a positive answer to the above question. We note that the removal of complementary handles is always possible in the case of maps of codimension at least $3$, but for codimension-$2$ maps it is not; see \cite{borpow}. A suggestive evidence to why this might still hold for the concordance of link immersions constructed in the proof of proposition \ref{prop:link} is the following.

\begin{rmk}
  Némethi and Pintér \cite{associmm} showed several restrictions on immersions $S^3\imto S^5$ which should be satisfied if they are link immersions of finitely $\AA$-determined germs. Many of these restrictions are versions of saying that link immersions are ``not very knotted''; in particular, they proved that if a link immersion is an embedding, then it can only be the standard embedding $S^3\subset S^5$ up to isotopy. If the answer to question \ref{ques:link} was positive, then corollary \ref{crly:no2h} would also show the ``simplicity'' of link immersions as noted in example \ref{ex:sphno2h}. 
  It seems to be an interesting problem to clarify and determine exactly ``how knotted'' such an immersion can be.
\end{rmk}

\end{document}